\documentclass[letterpaper,10pt,conference]{ieeeconf}

\usepackage{amsmath,amsfonts,amssymb,graphicx,graphics,enumerate,dsfont,float,color,bm}
\usepackage{filecontents}
\IEEEoverridecommandlockouts
\newcommand{\bbA}{\mathbb{A}}
\newcommand{\bbR}{\mathbb{R}}

\newcommand\numberthis{\addtocounter{equation}{1}\tag{\theequation}}

\def\DelX0{\Delta(X_0)}
\def\delpi0{\delta^{\pi}}

\newcommand{\argmax}{\mathop{\mathrm{argmax}}}

\newtheorem{Theorem}{Theorem}
\newtheorem{Remark}{Remark}

\newcommand{\comment}[1]{}

\newcommand{\ap}[1]{{\color{black}{#1}}}

\title{\LARGE \bf Bounding the Greedy Strategy in Finite-Horizon String Optimization}

\author{Yajing Liu, Edwin K. P. Chong, and Ali Pezeshki
\thanks{This work is supported in part by NSF under award CCF-1422658, and by the CSU Information Science and Technology Center (ISTeC).}%
\thanks{Y. Liu is with the Department of Electrical and Computer Engineering, Colorado State University, Fort Collins, CO 80523, USA {\tt\small yajing.liu@ymail.com}}%
\thanks{E. K. P. Chong and A. Pezeshki are with the Department of Electrical and Computer Engineering, and the Department of Mathematics, Colorado State University, Fort Collins, CO 80523, USA {\tt\small Edwin.Chong,Ali.Pezeshki@Colostate.Edu}}%
}

\begin{document}

\maketitle
\thispagestyle{empty}
\pagestyle{empty}

\begin{abstract}
We consider an optimization problem where the decision variable is a string of bounded length. For some time there has been an interest in bounding the performance of the greedy strategy for this problem. Here, we provide weakened sufficient conditions for the greedy strategy to be bounded by a factor of $(1-(1-1/K)^K)$, where $K$ is the optimization horizon length. Specifically, we introduce the notions of $K$-submodularity and $K$-GO-concavity, which together are sufficient for this bound to hold. By introducing a notion of \emph{curvature} $\eta\in(0,1]$, we prove an even tighter bound with the factor $(1/\eta)(1-e^{-\eta})$. Finally, we illustrate the strength of our results by considering two example applications. We show that our results provide weaker conditions on parameter values in these applications than in previous results.
\end{abstract}

\section{Introduction}

In a great number of problems in engineering and applied science,
we are faced with optimally choosing a string
(finite sequence) of actions over a finite horizon to maximize an
objective function. The problem arises in sequential decision making
in engineering, economics, management science, and medicine. To formulate the problem precisely, let $\bbA$ be a set of possible actions. At each stage $i$, we
choose an action $a_i$ from $\bbA$. Let $A=(a_1,a_2,\ldots,a_k)$
denote a \emph{string} of actions taken over $k$ consecutive
stages, where $a_i\in \bbA$ for $i=1,2,\ldots, k$. Let $\bbA^*$
denote the set of all possible strings of actions (of arbitrary
length, including the empty string $\varnothing$). Finally, let $f: \bbA^*\to \bbR$ be an objective function, where $\bbR$ denotes the real numbers. Our goal is to find a string $M\in \bbA^*$, with a length $|M|$ not larger than $K$ (prespecified), to maximize the objective function:
\begin{align}
\text{maximize } & f(M) \nonumber \\ 
\textrm{subject to } & M\in\bbA^*,\ |M|\leq K.
\label{eqn:2}
\end{align}

The solution to \eqref{eqn:2}, which we call the \emph{optimal
strategy}, is hard to compute in general. One approach is to use
dynamic programming via Bellman's principle (see, e.g.,
\cite{Bertsekas2000} and \cite{Pow07}). However, the computational complexity of
this approach grows exponentially with the size of $\bbA$
and the horizon length $K$. On the other hand, the \emph{greedy
strategy}, though suboptimal in general, is easy to compute because at each stage, we only have to find an action to maximize the step-wise gain in the objective function. But how does the greedy strategy compare with the optimal strategy in terms of the objective function? 

The above question has attracted widespread interest, with some key results in the context of \emph{string-submodularity} (see, e.g., \cite{streeter2008online}, \cite{alaei2010maximizing}, \cite{ZhC13J}, and \cite{ZhW13}). These papers extend the celebrated results of Nemhauser \emph{et al.}~\cite{nemhauser1978}, \cite{nemhauser19781}, and some further extensions of them (see, e.g., \cite{conforti1984submodular}, \cite{vondrak2010submodularity}, \cite{calinescu2011maximizing}, and \cite{wang2012}), on bounding the performance of greedy strategies in maximizing submodular functions over sets, to problem \eqref{eqn:2} that involves maximizing an objective function over strings.  In particular, Streeter and Golovin~\cite{streeter2008online} show that if, in \eqref{eqn:2}, the objective function $f$ is \emph{prefix} and \emph{postfix} monotone and has the diminishing-return property, then the greedy strategy achieves at least a $(1-e^{-1})$-approximation of the optimal strategy. Zhang~\textit{et~al.}~\cite{ZhC13J},~\cite{ZhW13} consider a weaker notion of the postfix monotoneity and provide sufficient conditions for the greedy strategy to achieve a factor of at least $(1-(1-1/K)^K)$, where $K$ is the optimization horizon length, of the optimal objective value. They also introduce several notions of curvature, with which the performance bound for the greedy strategy can be further sharpened. 

But all the sufficient conditions obtained so far involve strings of length greater than $K$, even though \eqref{eqn:2} involves only strings up to length $K$. This motivates a weakening of these sufficient conditions to involve only strings of length at most $K$, but still preserving the bounds here. 

In this paper, we introduce the notions of $K$-submodularity and $K$-GO-concavity, which together are sufficient for the $(1-(1-1/K)^K)$ bound to hold. By introducing a notion of \emph{curvature} $\eta\in(0,1]$, we prove an even tighter bound with the factor $(1/\eta)(1-e^{-\eta})$. Finally, we illustrate the strength of our results by considering two example applications. We show that our results provide weaker conditions on parameter values in these applications than in previous results reported in \cite{ZhC13J} and \cite{ZhW13}.


\section{Review of Related Work}\label{sc:II}

In this section, we first introduce some definitions related to strings and curvature. We then review the main results from \cite{ZhC13J} and \cite{ZhW13}. Specifically, the results there provide sufficient conditions on
the objective function $f$ in (\ref{eqn:2}) such that the greedy strategy achieves a $(1-(1-1/K)^K)$-bound.



\subsection{Strings and Curvature}

For a given string $A=(a_1,a_2,\ldots,a_k)$, we define its \emph{length} as $k$, denoted $|A|=k$. If $M=(a_1^m,a_2^m,\ldots, a_{k_1}^m)$ and $N=(a_1^n,a_2^n,\ldots, a_{k_2}^n)$ are two strings in $\bbA^*$, we write $M=N$ if $|M|=|N|$ and $a_i^m=a_i^n$ for each $i=1,2,\ldots, |M|$. Moreover, we define string \emph{concatenation} as $M\oplus N= (a_1^m,a_2^m,\ldots, a_{k_1}^m,a_1^n,a_2^n,\ldots, a_{k_2}^n)$. If $M$ and $N$ are two strings in $\bbA^*$, we write $M\preceq N$ if we have $N=M\oplus L$ for some $L\in \bbA^*$. In this case, we also say that $M$ is a \emph{prefix} of $N$.

A function from strings to real numbers, $f: \bbA^*\to \bbR$, is \emph{string submodular} if
\begin{itemize}
\item[i.] $f$ has the \emph{prefix-monotone} property: $ \forall M, N \in \bbA^*,$  $f(M\oplus N)\geq f(M)$.
\item[ii.] $f$ has the \emph{diminishing-return} property: $\forall M\preceq N \in \bbA^*, \forall a\in \bbA$, $f(M\oplus (a))-f(M) \geq f(N\oplus (a))-f(N)$.
\end{itemize}

A function from strings to real numbers, $f: \bbA^*\to \bbR$, is \emph{postfix monotone} if
 $$\forall M, N \in \bbA^*, f(M\oplus N)\geq f(N).$$

The \emph{total backward curvature of $f$} is defined as 
\begin{align*}
\sigma&=\max\limits_{a\in A, M\in \mathbb{A}^*}\\
&\left\{\frac{(f((a))-f(\emptyset))-(f((a)\oplus M)-f(M))}{f((a))-f(\emptyset)}\right\}.
\end{align*}

\subsection{Bounds for the Greedy Strategy}

We now define optimal and greedy strategies for problem \eqref{eqn:2} and some related notation.
\begin{itemize}
\item[(1)] \emph{Optimal strategy}: 
Any solution to \eqref{eqn:2} is called an \emph{optimal} strategy. If $f$ is prefix monotone, then there exists an optimal strategy with length $K$, denoted $O_K=(o_1,\ldots,o_K)$. Let $O_i=(o_1,\ldots, o_i)$ for $i=1, \ldots, K$.
\item[(2)] \emph{Greedy strategy}:
A string $G_{k}=(g_1,g_2,\ldots,g_{k})$ is called a \emph{greedy} strategy if $\forall i=1,2,\ldots,k$,
\begin{align*}
g_i&\in\mathop{\argmax}\limits_{g\in \bbA} f((g_1,g_2,\ldots,g_{i-1},g)).
\end{align*}
Let $G_i=(g_1,\ldots, g_i)$ for $i=1, \ldots, K$. 
\end{itemize} 

The following two theorems summarize the performance bounds 
in \cite{ZhC13J} and \cite{ZhW13}.
\begin{Theorem}
\label{thm:bounds1}
If $f$ is string submodular and $f(G_i\oplus O_K)\geq f(O_K)$ holds for all $i=1,\ldots, K-1$, then any greedy strategy $G_K$ satisfies 
\[
f(G_K) \geq \left(1-\left(1-\frac{1}{K}\right)^K\right)f(O_K)>(1-e^{-1})f(O_K).
\]
\end{Theorem}
\begin{Theorem}
\label{thm:bounds2}
If $f$ is string submodular and postfix monotone, then any greedy strategy $G_K$ satisfies 
\begin{align*}
f(G_K) &\geq \frac{1}{\sigma}\left(1-\left(1-\frac{\sigma}{K}\right)^K\right)f(O_K)\\
&> \frac{1}{\sigma}(1-e^{-\sigma})f(O_K)\\
&> (1-e^{-1})f(O_K).
\end{align*}
\end{Theorem}


Under additional assumptions on the curvature $\sigma$ of $f$, \cite{ZhC13J} and \cite{ZhW13} provide even tighter bounds. Notice that the sufficient conditions above involve strings of length greater than $K$, even though the problem (\ref{eqn:2}) involves only strings up to length $K$. This motivates a weakening of these sufficient conditions to involve only strings of length at most $K$, but still preserving the bounds here. In the next section, we present our main results along these lines. In Section~\ref{sc:IV}, we show that these weakened sufficient conditions also lead to weaker requirements than in \cite{ZhC13J} and \cite{ZhW13} for two application examples.

\section{Main Results}\label{sc:III}

Before stating our main results, we first introduce some definitions on
$f: \bbA^*\to \bbR$.
\begin{itemize}
\item[i] $f$ is \emph{$K$-monotone} if
$\forall M, N \in \bbA^*,$ and $|M|+|N|\leq K$,  $f(M\oplus N)\geq f(M)$.

\item[ii.] $f$ is $K$-\emph{diminishing} if
$\forall M\preceq N \in \bbA^*$ and $|N|\leq K-1$, $\forall a\in \bbA$, $f(M\oplus (a))-f(M) \geq f(N\oplus (a))-f(N)$.

\item[iii.] $f$ is $K$-\emph{submodular} if it is both $K$-monotone and $K$-diminishing. 

\item[iv.] Let $G_i=(g_1,\ldots, g_i)$ (as before) and $\bar{O}_{K-i}=(o_{i+1},\ldots, o_K)$ for $i=1,\ldots, K$. Then, $f$ is $K$-\emph{GO-concave} if for $1\leq i\leq K-1$, 
\[
f(G_i\oplus \bar{O}_{K-i})\geq \frac{i}{K}f(G_i) + \left(1-\frac{i}{K}\right)f(O_K).
\]
\end{itemize}

Notice that these definitions involve only strings of length at most $K$.
Moreover, it is clear that if $f$ is string submodular, prefix monotone, and has the diminishing-return property, then $f$ is string $K$-submodular, $K$-monotone, and $K$-diminishing. Under these weaker conditions, we show that the previous bounds on the greedy strategy still hold.

\begin{Theorem}
\label{thm:myopicbounds3}
If $f$ is $K$-submodular and $K$-GO-concave,
then 
\[
f(G_K) \geq \left(1-\left(1-\frac{1}{K}\right)^K\right)f(O_K)>(1-e^{-1})f(O_K).
\]
\end{Theorem}
\begin{proof}
Because $f$ is $K$-diminishing, we have that for $1\leq i\leq K$,
\[
f(o_i)\geq f(O_i)-f(O_{i-1}).
\]
By definition of the greedy strategy, for $1\leq i\leq K$,
\[
f(G_1)\geq f(o_i)\geq f(O_i)-f(O_{i-1}).
\]
Summing the inequality above over $i$ from $1$ to $K$ produces
\begin{align*}
\sum\limits_{i=1}^Kf(G_1)&\geq \sum\limits_{i=1}^K(f(O_i)-f(O_{i-1}))\\
\Rightarrow\quad Kf(G_1)&\geq f(O_K)\\
\Rightarrow\quad f(G_1)&\geq \frac{1}{K}f(O_K).
\end{align*}
For $1\leq i\leq K-1$, because $f$ is $K$-diminishing, we have 
\begin{align*}
f&(G_i\oplus o_K)-f(G_i)\\
&\geq f(G_i\oplus (o_{i+1},\cdots, o_K))-f(G_i\oplus(o_{i+1},\cdots, o_{K-1})),\\
f&(G_i\oplus o_{K-1})-f(G_i)\\
&\geq f(G_i\oplus (o_{i+1},\cdots, o_{K-1}))-f(G_i\oplus(o_{i+1},\cdots, o_{K-2})),\\
&\vdots\\
f&(G_i\oplus o_{i+2})-f(G_i)\\
&\geq f(G_i\oplus (o_{i+1},o_{i+2}))-f(G_i\oplus(o_{i+1})),\\
f&(G_i\oplus o_{i+1})-f(G_i)\geq f(G_i\oplus (o_{i+1}))-f(G_i)).
\end{align*}
Summing the inequalities above, we have 
\[\sum\limits_{j={i+1}}^K(f(G_i\oplus o_j)-f(G_i))\geq f(G_i\oplus\bar{O}_{K-i})-f(G_i).\]
By definition of the greedy strategy, we have for $i+1\leq j\leq K,$
\[f(G_{i+1})-f(G_i)\geq f(G_i\oplus o_j)-f(G_i),\] which implies that
\[(K-i)(f(G_{i+1})-f(G_i))\geq \sum\limits_{j={i+1}}^K(f(G_i\oplus o_j)-f(G_i)).\]
Hence, we have for $1\leq i\leq K-1$,
\begin{equation}
\label{eqn:3}
f(G_{i+1})-f(G_i)\geq\frac{1}{K-i} (f(G_i\oplus\bar{O}_{K-i})-f(G_i)).
\end{equation}
By $K$-GO-concavity, for $1\leq i\leq K-1$ we have 
\begin{align*}
f(G_{i+1})&-f(G_i)\\
&\geq \frac{1}{K-i}(f(G_i\oplus \bar{O}_{K-i})-f(G_i))\\
&\geq \frac{1}{K-i}\left(\frac{K-i}{K}f(O_K)+\frac{i}{K}f(G_i)-f(G_i)\right)\\
&= \frac{1}{K}(f(O_K)-f(G_i)),
\end{align*} from which we get
\[
f(G_{i+1}) \geq \frac{1}{K} f(O_K) + \left(1-\frac{1}{K}\right) f(G_i).
\]
Therefore,
\begin{align*}
f(G_K)&\geq \frac{1}{K}f(O_K)+\left(1-\frac{1}{K}\right)f(G_{K-1})\\
&\quad\vdots \\
&\geq \frac{1}{K}f(O_K)\sum\limits_{i=0}^{K-1}\left(1-\frac{1}{K}\right)^i\\
&=\left(1-\left(1-\frac{1}{K}\right)^K\right)f(O_K).
\end{align*}
Because $1-\left(1-\frac{1}{K}\right)^K\searrow 1-e^{-1}$ as $K\rightarrow \infty$, we also have
\[
f(G_K) \geq \left(1-\left(1-\frac{1}{K}\right)^K\right)f(O_K)>(1-e^{-1})f(O_K).
\]
\end{proof}

Next, we introduce a new notion of curvature $\eta$ as follows:
\begin{align*}
\eta&=\max\limits_{1\leq i\leq K-1}\\
&\left\{\frac{Kf(G_i)-(Kf(G_i\oplus \bar{O}_{K-i})-(K-i)f(O_K))}{(K-i)f(G_i)}\right\}.
\end{align*}
If $f$ is $K$-GO-concave, then for $1\leq i\leq K-1$ we have 
\begin{align*}
Kf(G_i)-&(Kf(G_i\oplus \bar{O}_{K-i})-(K-i)f(O_K))\\
&\leq Kf(G_i)- if(G_i)\\
&=(K-i)f(G_i),
\end{align*}
which implies that $\eta\leq 1$.
The following theorem gives a bound related to the curvature $\eta$.
\begin{Theorem}
\label{thm:myopicbounds4}
If $f$ is $K$-submodular and $K$-GO-concave, then
\begin{align*}
f(G_K) &\geq \frac{1}{\eta}\left(1-\left(1-\frac{\eta}{K}\right)^K\right)f(O_K)\\
&>\frac{1}{\eta}(1-e^{-\eta})f(O_K).
\end{align*}
\end{Theorem}

\begin{proof}
By definition of the curvature $\eta$, we have 
\[
f(G_i\oplus \bar{O}_{K-i})-f(G_i)\geq \frac{K-i}{K}(f(O_K)-\eta f(G_i)).
\]
By definition of the greedy strategy and inequality (\ref{eqn:3}),

 we have 
\begin{align*}
f(G_{i+1})-f(G_i)&\geq \frac{1}{K-i}(f(G_i\oplus \bar{O}_{K-i})-f(G_i))\\
&\geq\frac{1}{K-i}\cdot\frac{K-i}{K}(f(O_K)-\eta f(G_i))\\
&= \frac{1}{K}(f(O_K)-\eta f(G_i)),
\end{align*}
from which we get
\[
f(G_{i+1}) \geq \frac{1}{K} f(O_K) + \left(1-\frac{\eta}{K}\right) f(G_i).
\]
Therefore,
\begin{align*}
f(G_K)&\geq \frac{1}{K}f(O_K)+\left(1-\frac{\eta}{K}\right)f(G_{K-1})\\
&\quad\vdots \\
&\geq \frac{1}{K}f(O_K)\sum\limits_{i=0}^{K-1}\left(1-\frac{\eta}{K}\right)^i\\
&=\frac{1}{\eta}\left(1-\left(1-\frac{\eta}{K}\right)^K\right)f(O_K).
\end{align*}
Because $\frac{1}{\eta}\left(1-\left(1-\frac{\eta}{K}\right)^K\right)\searrow \frac{1}{\eta}(1-e^{-\eta})$ as $K\rightarrow \infty$, we also have
\begin{align*}
f(G_K) &\geq \frac{1}{\eta}\left(1-\left(1-\frac{\eta}{K}\right)^K\right)f(O_K)\\
&>\frac{1}{\eta}(1-e^{-\eta})f(O_K).
\end{align*}
\end{proof}

\emph{Remarks}

\begin{itemize}
\item The term $\frac{1}{\eta}(1-e^{-\eta})$ is decreasing in $\eta\in(0,1]$.
\item When $\eta=1$, $\frac{1}{\eta}(1-e^{-\eta})=1-e^{-1}$, which is the bound in Theorem \ref{thm:myopicbounds3}.
Moreover, for $0<\eta<1$, $\frac{1}{\eta}(1-e^{-\eta})>1-e^{-1}$.
Hence, Theorem \ref{thm:myopicbounds4} is a generalization of Theorem \ref{thm:myopicbounds3} and gives a tighter bound.
\item When $\eta\rightarrow 0,$ we have $\frac{1}{\eta}(1-e^{-\eta})\rightarrow 1$, making the greedy strategy asymptotically optimal.
\end{itemize}

\section{Applications}\label{sc:IV}

In this section, we \ap{consider two example applications, namely task assignment and adaptive measurement design, to illustrate the strength of our results. In each case, we derive sufficient conditions, on the parameter values of the problem, for the greedy strategy to achieve the $(1-(1-1/K)^K)$ bound. These sufficient conditions are weaker than those we previously reported in \cite{ZhC13J}.}

\subsection{Task Assignment Problem}

\ap{As our first example application, we consider the task
assignment problem that was posed in \cite{streeter2008online} and
was further analyzed in \cite{ZhC13J}. In this problem, we have $n$
subtasks and a set $\mathbb{A}$ of $K$ agents. At each stage, we get
to assign a subtask to an agent, who accomplishes the task with some
probability. Let $p_{i}^j(a)$ denote the probability of
accomplishing subtask $i$ at stage $j$ when it is assigned to agent
$a\in \mathbb{A}$. Assume that $p_{i}^{j}(a)\in [L_i(a),U_i(a)]$,
$0<L_i(a)<U_i(a)< 1$, and that the limits of the interval are
independent of the stage in which subtask $i$ is assigned to agent
$a$. Let $X_i(a_1,a_2,\ldots,a_k)$ denote the random variable that
describes whether or not subtask $i$ has been accomplished after the
sequence of assignments $(a_1,a_2,\ldots,a_k)$ over $k$ steps. Then,
$\frac{1}{n}\sum_{i=1}^{n} X_i(a_1,a_2,\ldots,a_k)$ is the fraction
of subtasks accomplished by employing agents
$(a_1,a_2,\ldots,a_k)$ over $k$ steps. The objective function $f$
for this problem is the expected value of this fraction, which can be written as 
\begin{align*}
f((a_1,\ldots, a_k))=\frac{1}{n}\sum\limits_{i=1}^n\left(1-\prod\limits_{j=1}^k\left(1-p_i^j(a_j)\right)\right).
\end{align*} 
We wish to derive sufficient conditions on the set of parameters
$\{(L(a),U(a)) \ | a\in \mathbb{A}\}$ so that $f$ is $K$-monotone, $K$-diminishing, and $K$-GO-concave.

For simplicity, we consider the case of $n=1$. But our results can easily be generalized to the case where $n>1$. For $n=1$, the objective function $f$ reduces to}   
 \begin{equation}\label{eq:Ex1fn1}
f((a_1,\ldots, a_k))=1-\prod\limits_{j=1}^k(1-p_1^j(a_j)),
\end{equation} 
\ap{and from here on we simply use $p^j(a_j)$ in place of $p_1^j(a_j)$.} 

It is easy to check that $f$ is \ap{$K$-monotone. For $f$ to be $K$-diminishing, it suffices to have}
\[
f(M\oplus (a))-f(M)\geq f(M\oplus (b)\oplus (a))-f(M\oplus (b)),
\]
\ap{for any $a,b\in\mathbb{A}$ and for any $M\in\mathbb{A}^*$ with $|M|\le K-2$.} 
Let $M=(a_1,\ldots, a_m)$, then we have  
\[p^{m+1}(a)\geq (1-p^{m+1}(b))p^{m+2}(a).\]
Suppose that $L(a)\leq p^j(a)\leq U(a)$ for all $a\in \mathbb{A}, j=1,2,\ldots, K.$ Let 
\[
\hat{U}=\max\limits_{a\in\mathbb{A}}U(a)\]
and 
\[
\hat{L}=\min\limits_{a\in\mathbb{A}}L(a).
\]
Then, we can write 
\[
p^{m+1}(a)\geq L(a)\geq \hat{L}
\]
and 
\[
(1-p^{m+1}(b))p^{m+2}(a)\leq (1-L(b))U(a)\leq (1-\hat{L})\hat{U}.
\] 
\ap{Thus, a sufficient condition for $f$ to be $K$-diminishing is}
\begin{equation}\label{eqn:4}
\hat{L}\geq (1-\hat{L})\hat{U}.
\end{equation}

\ap{Now, let us rearrange the $K$-GO-concavity condition as}
\begin{align*}
(K-i)(f(O_K)-f(G_i\oplus \bar{O}_{K-i})) \quad\quad\quad \\
\leq i(f(G_i\oplus \bar{O}_{K-i})-f(G_i)).
\end{align*}
\ap{Replacing for $f$ from \eqref{eq:Ex1fn1} gives (after simplifying)} 
\begin{align*}
&(K-i)\prod_{j=i+1}^K(1-p^j(o_i))\left[1-\frac{\prod_{j=1}^i(1-p^j(o_i))}{\prod_{j=1}^i(1-p^j(g_j))}\right]\\
&\leq i\left[1-\prod_{j=i+1}^K(1-p^j(o_i))\right].
\end{align*}
Because $f(O_K)\geq f(G_i\oplus \bar{O}_{K-i})$, we have 
\ap{
\[
\frac{\prod_{j=1}^i(1-p^j(o_i))}{\prod_{j=1}^i(1-p^j(g_j))}\leq 1.
\]
Therefore, to have $K$-GO-concavity it suffices to have} 
\[
(K-i)\prod_{j=i+1}^K(1-p^j(o_i))\leq
i\left[1-\prod_{j=i+1}^K(1-p^j(o_i))\right],
\]
\ap{or equivalently}
\begin{equation}\label{eq:Ex1GO}
\prod_{j=i+1}^K(1-p^j(o_i))\leq \frac{i}{K}, 
\end{equation}
for $1\leq i\leq K-1$. If we assume that $p^i(o_i)\geq \frac{1}{i}$ for $2\leq i\leq K$, then it is easy to see \ap{that \eqref{eq:Ex1GO}} holds for $1\leq i\leq K-1.$ \ap{Thus, a sufficient condition for $K$-GO-concavity is} 
\begin{equation}
\label{eqn:Ex1Suff}
\hat{L}\geq \frac{1}{2}.
\end{equation}
\ap{If \eqref{eqn:Ex1Suff} holds then \eqref{eqn:4} also holds. Thus, \eqref{eqn:Ex1Suff} is sufficient for the greedy strategy to achieve the $(1-(1-\frac{1}{K})^K)$ bound.}

\ap{
\begin{Remark}
The sufficient condition in \cite{ZhC13J} requires \eqref{eqn:4} and 
\begin{equation}\label{eqn:Ex1Extra}
p^1(g_1)\geq 1-c^K,
\end{equation}
where 
\[
c=\min\limits_{a\in\mathbb{A}}\frac{1-U(a)}{1-L(a)}.
\]
When all $p^j(a_j)\ge 1/2$, then \eqref{eqn:Ex1Suff} and \eqref{eqn:4} automatically hold, but \eqref{eqn:Ex1Extra} is not necessarily satisfied. In that sense, the $K$-monotone, $K$-diminishing, and $K$-Go concavity conditions are weaker sufficient conditions for achieving the $(1-(1-\frac{1}{K})^K)$ bound than the prefix monotone, diminishing-return, and postfix monotone conditions of \cite{ZhC13J}.
\end{Remark}
}


%
%
%
%

\subsection{Adaptive Measurement Problem}

\ap{As our second example application, we consider the adaptive measurement design problem posed in \cite{LiC12} and \cite{ZhC13J}. Consider a signal of interest $x \in \mathbb{R}^2$ with normal prior distribution $\mathcal{N} (0, I)$, where $I$ is the $2\times 2$ identity matrix; our analysis easily generalizes to dimensions larger than $2$. Let $\mathbb{A}=\{\mathrm{Diag}\left(\sqrt{e},\sqrt{1-e}\right) \ : \ e\in[0.5,1]\}$. At each stage $i$, we make a measurement $y_i$ of the form} 
\[
y_i=A_ix+w_i, 
\]
\ap{where $A_i \in\mathbb{A}$ and $w_i$ is a Gaussian measurement noise vector with mean zero and covariance $\sigma_i^2 I$. 

The objective is to choose a string of measurement matrices $\{A_i\}_{i=1}^k$ with $k\le K$ to maximize the information gain:
\[
f((a_1,\ldots, a_k))=H_0-H_k.
\]
Here $H_0=\frac{N}{2}\text{log}(2\pi e)$ is the entropy of the prior
distribution of $x$ and $H_k$ is the entropy of the posterior
distribution of $x$ given $\{y_i\}_{i=1}^k$; that is,
\[
H_k=\frac{1}{2}\text{log det}(P_k)+\frac{N}{2}\text{log}(2\pi e),
\] 
where 
\[
P_k=\left(P_{k-1}^{-1}+\frac{1}{\sigma_k^2}A_k^TA_k\right)^{-1}
\]
is the posterior covariance of $x$ given $\{y_i\}_{i=1}^k$ \cite{LiC12}.

We wish to derive sufficient conditions on the set of parameters
$\{\sigma_i^2\}_{i=1}^K$ so that $f$ is $K$-monotone,
$K$-diminishing, and $K$-GO-concave.} It is easy to see that $f$ is
\ap{$K$-monotone} by form, and it is \ap{$K$-diminishing} if
\ap{$\{\sigma_i^2\}_{i=1}^{K}$ is a non-decreasing sequence, that is,
\begin{equation}
\label{eqn:6}
\sigma_{i+1}^2\geq \sigma_i^2, \ \ \mathrm{for} \ i=1,2,\ldots, K-1.
\end{equation}

Let $A_i^g=\text{Diag}(\sqrt{e_i}, \sqrt{1-e_i})$ and
$A_i^*=\text{Diag}(\sqrt{e_i^*}, \sqrt{1-e_i^*})$ be the greedy and
optimal actions at stage $i$, respectively; that is, $g_i=A_i^g$ and $o_i=A_i^*$. Then, the $K$-GO-concavity condition for this problem is that} for $1\leq i \leq K-1$, we must have 
\begin{align*}
&(S_{i}^*+\bar{S}_{K-i}^*)^{K-i}(c_K-(S_{i}^*+\bar{S}_{K-i}^*))^{K-i}S_i^i(a_i-S_i)^i\\
&\leq (S_{i}+\bar{S}_{K-i}^*)^K(c_K-(S_{i}+\bar{S}_{K-i}^*))^K,\numberthis\label{eqn:7}
\end{align*}
where 
\begin{align*}
S_{i}^*=&1+\sum\limits_{j=1}^{i}\frac{1}{\sigma_j^2}e_j^*,\\
\bar{S}_{K-i}^*=&\sum\limits_{j=i+1}^{K}\frac{1}{\sigma_j^2}e_j^*,\\
S_{i}=&1+\sum\limits_{j=1}^{i}\frac{1}{\sigma_j^2}e_j,\\
a_i=&2+\sum\limits_{j=1}^i\frac{1}{\sigma_j^2},\\ 
c_K=&2+\sum\limits_{j=1}^K\frac{1}{\sigma_j^2}.
\end{align*}
\ap{Because} $f(O_K)\geq f(G_i\oplus \bar{O}_{K-i})$, \ap{we have}
\begin{align*}
(S_{i}^*+\bar{S}_{K-i}^*)(c_K-(S_{i}^*+\bar{S}_{K-i}^*))\quad\quad\quad\\
\geq (S_{i}+\bar{S}_{K-i}^*)(c_K-(S_{i}+\bar{S}_{K-i}^*)).
\end{align*}
It is easy to check that 
\[
S_i(a_i-S_i)\leq (S_{i}+\bar{S}_{K-i}^*)(c_K-(S_{i}+\bar{S}_{K-i}^*)).
\]
Therefore, we have 
\begin{equation}
\label{eqn:8}
S_i(a_i-S_i)\leq (S_{i}^*+\bar{S}_{K-i}^*)(c_K-(S_{i}^*+\bar{S}_{K-i}^*)).
\end{equation}

Let \[g(i)=(S_{i}^*+\bar{S}_{K-i}^*)^{K-i}(c_K-(S_{i}^*+\bar{S}_{K-i}^*))^{K-i}S_i^i(a_i-S_i)^i.\] 
Then,
\[\frac{g(i+1)}{g(i)}=\frac{S_i(a_i-S_i)}{(S_{i}^*+\bar{S}_{K-i}^*)(c_K-(S_{i}^*+\bar{S}_{K-i}^*))}.\]
\ap{By (\ref{eqn:8}), $g(i)$} is non-increasing. Hence, \ap{it suffices to have}   
\begin{align*}
&(S_{1}^*+\bar{S}_{K-1}^*)^{K-1}(c_K-(S_{1}^*+\bar{S}_{K-1}^*))^{K-1}S_1(a_1-S_1)\\
&\leq (S_{1}+\bar{S}_{K-1}^*)^K(c_K-(S_{1}+\bar{S}_{K-1}^*))^K
\numberthis\label{eq:9}
\end{align*}
in order to \ap{get $K$-GO-concavity.}


Let $T_1=a_1-S_1$, $T_1^*=a_1-S_1^*,$ and $\bar{T}_{K-1}^*=\sum\limits_{j=2}^K\frac{1-e_j^*}{\sigma_j^2}$. Then, we can rewrite (\ref{eq:9}) as 
\begin{align*}
(S_{1}^*+\bar{S}_{K-1}^*)^{K-1}(T_1^*+\bar{T}_{K-1}^*))^{K-1}S_1T_1 \quad\quad\quad\\
\leq (S_{1}+\bar{S}_{K-1}^*)^K(T_1+\bar{T}_{K-1}^*)^K.
\numberthis \label{eqn:11}
\end{align*}
\ap{If}  
\begin{align}
(S_{1}^*+\bar{S}_{K-1}^*)(T_1^*+\bar{T}_{K-1}^*))\quad\quad\quad\nonumber\\
= (S_{1}+\bar{S}_{K-1}^*)(T_1+\bar{T}_{K-1}^*)),
\label{eqn:12}
\end{align}
\ap{that is, to have $f(O_1)=f(G_1)$,  then (\ref{eqn:11}) always holds, because} $S_1T_1\leq (S_{1}+\bar{S}_{K-1}^*)(T_1+\bar{T}_{K-1}^*)).$ 

\ap{We now show that \eqref{eqn:12} always holds, given the action set $\mathbb{A}$ considered in this example. In other words, the $K$-GO-concavity condition is satisfied and this means that \eqref{eqn:6} is a sufficient condition for achieving the $(1-(1-\frac{1}{K})^K)$} bound.  


By definition of the greedy strategy, we have $f(G_1)\geq f(O_1)$, which means
\begin{align*}
\left(1+\frac{1}{\sigma_1^2}e_1\right)\left(1+\frac{1}{\sigma_1^2}(1-e_1)\right) \quad\quad \quad\\
\geq
\left(1+\frac{1}{\sigma_1^2}e_1^*\right)\left(1+\frac{1}{\sigma_1^2}(1-e_1^*)\right).
\end{align*}
\ap{Simplifying the above inequality gives}
\begin{equation}\label{eqn:13}
(e_1-e_1^*)(1-(e_1+e_1^*))\geq 0.
\end{equation}
\ap{Because} $f(O_K)\geq f(G_1\oplus \bar{O}_{k-1})$, \ap{we have}  
\begin{align*}
\left(1+\frac{e_1^*}{\sigma_1^2}+\sum\limits_{j=2}^K\frac{e_j^*}{\sigma_j^2}\right)&\left(1+\frac{(1-e_1^*)}{\sigma_1^2}+\sum\limits_{j=2}^K\frac{(1-e_j^*)}{\sigma_j^2}\right)\\
&\geq\\
\left(1+\frac{e_1}{\sigma_1^2}+\sum\limits_{j=2}^K\frac{e_j^*}{\sigma_j^2}\right)&\left(1+\frac{(1-e_1)}{\sigma_1^2}+\sum\limits_{j=2}^K\frac{(1-e_j^*)}{\sigma_j^2}\right),
\end{align*}
which implies that
\begin{align*}
(e_1-e_1^*)&\left[\sum\limits_{j=2}^K\frac{1}{\sigma_j^2}(2e_j^*-1)\right]\\
&\geq\frac{1}{\sigma_1^2}(e_1-e_1^*)(1-(e_1+e_1^*)).
\numberthis\label{eqn:14}
\end{align*}

The inequality (\ref{eqn:13}) implies that $e_1\leq e_1^*$.
From (\ref{eqn:14}) and (\ref{eqn:13}), we have that 
\[
(e_1-e_1^*)\left[\sum\limits_{j=2}^K\frac{1}{\sigma_j^2}(2e_j^*-1)\right]\geq
0,\]
which implies that $e_1=e_1^*$.
Since if $e_1\neq e_1^*$, then $e_1< e_1^*,$ which implies that
\[
(e_1-e_1^*)\left[\sum\limits_{j=2}^K\frac{1}{\sigma_j^2}(2e_j^*-1)\right]< 0,
\]
while 
\[
\frac{1}{\sigma_1^2}(e_1-e_1^*)\left(1-(e_1+e_1^*)\right)>0,
\]
which contradicts (\ref{eqn:14}). Hence, we have $e_1=e_1^*$, which means $G_1=O_1$, and the inequality \eqref{eqn:12} holds.

%

\ap{
\begin{Remark} The sufficient condition in \cite{ZhC13J} for achieving the $(1-(1-\frac{1}{K})^K)$ bound in this problem requires both \eqref{eqn:6} and 
\begin{equation}
\frac{b^{-2}}{a^{-2}-b^{-2}}\geq \frac{(2K-2)^2}{4}(a^{-2}+b^{-2})+1,
\label{eqn:17}
\end{equation}
where $[a,b]$ is the interval that contains all the $\sigma_i$s. Therefore, the condition derived in this paper is a weaker sufficient condition than that obtained in \cite{ZhC13J}.  
\end{Remark}
}

\section{Conclusion}\label{sc:V}

We considered an optimization problem (\ref{eqn:2}) where the decision variable is a string of length at most $K$. For this problem, we reviewed some previous results on bounding the greedy strategy. In particular, the results of \cite{ZhC13J} and \cite{ZhW13} provide sufficient conditions for the greedy strategy to be bounded by a factor of $(1-(1-1/K)^K)$. We then presented \emph{weakened} sufficient conditions for this same bound to hold, by introducing the notions of $K$-submodularity and $K$-GO-concavity. Next, we introduced a notion of \emph{curvature} $\eta\in(0,1]$, which furnishes an even tighter bound with the factor $\frac{1}{\eta}(1-e^{-\eta})$. Finally, we illustrated our results by considering two example applications. We showed that our new results provide weaker conditions on parameter values in these applications than in \cite{ZhC13J} and \cite{ZhW13}.

\bibliographystyle{IEEEbib}

\end{document}